\documentclass[letterpaper,11pt,oneside,reqno]{amsart}

\usepackage{amsmath,amssymb,amsthm,amsfonts}
\usepackage[colorlinks=true,linkcolor=blue,citecolor=red]{hyperref}
\usepackage{graphicx,color}
\usepackage{upgreek}
\usepackage[mathscr]{euscript}

\allowdisplaybreaks
\numberwithin{equation}{section}

\usepackage{tikz}
\usetikzlibrary{shapes,arrows,positioning,decorations.markings}

\usepackage{array}
\usepackage{adjustbox}
\usepackage{cleveref}
\usepackage{enumerate}

\usepackage[DIV=12]{typearea}

\newtheorem{proposition}{Proposition}[section]
\newtheorem{lemma}[proposition]{Lemma}
\newtheorem{corollary}[proposition]{Corollary}
\newtheorem{theorem}[proposition]{Theorem}
\theoremstyle{definition}

\newtheorem{remark}[proposition]{Remark}

\begin{document}
\title[Quenched Central Limit Theorem in a Corner Growth Setting]{Quenched Central Limit Theorem\\in a Corner Growth Setting}

\author[C. Gromoll]{H. Christian Gromoll}\thanks{C. Gromoll,
	University of Virginia, Department of Mathematics, 141 Cabell Drive, Kerchof
	Hall, P.O. Box 400137, Charlottesville, VA 22904,
	USA, E-mail: \texttt{gromoll@virginia.edu}}

\author[M. Meckes]{Mark W. Meckes}\thanks{M. Meckes, Department of Mathematics, Applied
	Mathematics, and Statistics, Case Western Reserve University, 10900 Euclid
	Ave., Cleveland, OH 44106, USA, E-mail: \texttt{mark.meckes@case.edu}}

\author[L. Petrov]{Leonid Petrov}\thanks{L. Petrov, University of Virginia,
	Department of Mathematics, 141 Cabell Drive, Kerchof Hall, P.O. Box 400137,
	Charlottesville, VA 22904, USA, and Institute for Information Transmission
	Problems, Bolshoy Karetny per. 19, Moscow, 127994,
	Russia, E-mail: \texttt{petrov@virginia.edu}}

\date{}

\begin{abstract}
  We consider point-to-point directed paths in a random environment on the
  two-dimensional integer lattice. For a general independent environment under
  mild assumptions we show that the quenched energy of a typical path
  satisfies a central limit theorem as the mesh of the lattice goes to zero.
  Our proofs rely on concentration of measure techniques and some
  combinatorial bounds on families of paths.  \end{abstract}

\maketitle

\section{Introduction and main results}
\label{sec:intro}

A number of well-known probabilistic models derive their underlying complexity
from a variant of the following simple setup. Put independent and identically
distributed weights at each vertex of the two-dimensional integer lattice.
Given a lattice point in the first quadrant, consider all paths in the lattice
from the origin to this point that only move up or to the right at each step.
Each such path has a random energy given by the sum of the weights along the
path, and so the collection of random energies indexed by the up-right paths
exhibits a complicated dependence structure. This dependence is at the heart
of the difficulty in understanding such models as the totally asymmetric
simple exclusion process (TASEP), the infinite tandem queue, the random
directed polymer, or the corner growth model. 

For example in the corner growth model, or directed nearest neighbor
last-passage percolation on the 2d lattice, the fundamental issue is to
understand the distribution of the maximum-energy path to a given point. This maximal energy
represents the last passage time to the point, or time at which it joins
the growing corner shape. In directed polymer models one assigns 
a Boltzmann weight to each path according to its random energy, and one is concerned with
the sum of all polymer weights, or partition function, which normalizes the
polymer weights into probabilities. (See e.g. \cite{CorwinKPZ} for a description
of the model as well as its relationship to equivalent models). Both are
difficult models because of the high degree of dependence in the joint
distribution of the collection of path energies. A more detailed discussion of
these models as they relate to our result appears further below. 

In this paper we derive a result about the joint distribution of the path
energies which to our knowledge has not been observed previously. Namely, we
show that conditional on the environment of weights, the empirical
distribution of the family of path energies is approximately Gaussian. More
precisely, for almost every environment, the energy of an up-right path
selected uniformly at random is asymptotically normally distributed as the
mesh gets small.

This is the content of our main theorem, which is proved for generally
distributed weights with nonzero variance and
under further moment assumptions (in fact, in our setting the 
weights need not even be identically distributed).  We now introduce some
notation, describe our results, and then discuss connections to the
corner growth and directed random polymer models.

\subsection{Up-right paths in a random environment}

We work inside the positive quadrant $\mathbb{Z}_{\ge1}^{2}$ of the
two-dimensional integer lattice. By convention, coordinates $(i,j)\in
	\mathbb{Z}_{\ge1}^{2}$ refer to squares, see \Cref{fig:coordinates_paths}. A
(fixed) \emph{environment} $w$ is an assignment of a real number $w_{ij}\in
	\mathbb{R}$ to every square of $\mathbb{Z}_{\ge1}^{2}$. We call the 
$w_{ij}$ \emph{weights}.

Fix integers $M,N\ge1$ and consider the rectangle $1\le i\le M$, $1\le j\le N$
inside the quadrant. Denote it by $\square_{M,N}$. The environment $w$
restricted to $\square_{M,N}$ is a vector in $\mathbb{R}^{MN}$.

An \emph{up-right path} $\sigma$ from $(1,1)$ to $(M,N)$ is a collection
$\left\{ (i_k,j_k)\colon k=1,\ldots,M+N-1  \right\}$ such that
$(i_{k+1}-i_k,j_{k+1}-j_k)$ is either $(1,0)$ (horizontal step) or $(0,1)$
(vertical step), $(i_1,j_1)=(1,1)$, and $(i_{M+N-1},j_{M+N-1})=(M,N)$. To each
up-right path $\sigma$ we associate a vector
\begin{equation*}
	Y^{\sigma}\in \mathbb{R}^{MN},\qquad
	Y^{\sigma}_{ij}:=\mathbf{1}_{(i,j)\in \sigma}, \qquad 1\le i\le M, \ 1\le j\le
	N.
\end{equation*}
Here and below $\mathbf{1}_{A}$ is the indicator of $A$.

For any up-right path $\sigma=\left\{ (i_k,j_k) \right\}$ define its
\emph{energy} with respect to an environment $w$ as
\begin{equation}
	\label{weight_of_path}
	\langle Y^{\sigma},w \rangle=\sum_{k=1}^{M+N-1}w_{i_kj_k},
\end{equation}
where $\langle \cdot,\cdot \rangle$ is the standard inner product in
$\mathbb{R}^{MN}$.

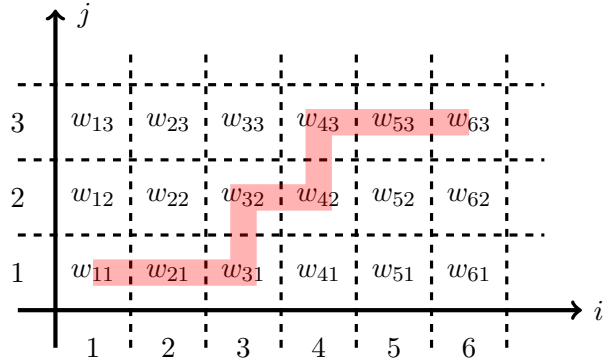
\begin{figure}[htpb]
	\begin{tikzpicture}[scale=1, very thick]
		\draw[->, ultra thick] (-.5,0)--++(7.5,0) node [right] {$i$};
		\draw[->, ultra thick] (0,-.5)--++(0,4.5) node [right,
			xshift=4, yshift=-2] {$j$};
		\foreach \xx in {1,...,6}
			{
				\draw[dashed] (\xx,-.5)--++(0,4);
				\node at (\xx-.5,-.5) {$\xx$};
			}
		\foreach \yy in {1,...,3}
			{
				\draw[dashed] (-.5,\yy)--++(7,0);
				\node at (-.5,\yy-.5) {$\yy$};
			}
		\foreach \xx in {1,...,6}
			{
				\foreach \yy in {1,...,3}
					{
						\node at (\xx-.5,\yy-.5)
						{$w_{\xx\yy}$};
					}
			}
		\draw[line width=10pt, red, opacity=.3]
		(.5,.5)--++(2,0)--++(0,1)--++(1,0)
		--++(0,1)--++(2,0);
	\end{tikzpicture}
	\caption{Environment and an up-right path from $(1,1)$ to
		$(M,N)=(6,3)$.}
	\label{fig:coordinates_paths}
\end{figure}

Now suppose both the environment and the path
are chosen at random, independently of each other. Assume the random
environment $w$ consists of independent random variables $w_{ij}$ defined on a
probability space $(\Omega_w,\mathcal{F}_w,\mathbb{P}_w)$ with
\begin{equation}
	\label{environment_moment_assumptions}
	\mathbb{E}_ww_{ij}=0,\qquad
	\mathbb{E}_ww_{ij}^2=1,\qquad
  \mathbb{E}_w|w_{ij}|^{\mathsf{p}}\le \mathsf{K},\qquad \text{for all $i,j$,}
\end{equation}
for some $\mathsf{p}$ (to be specified later) and $\mathsf{K}>0$. Here $\mathbb{E}_w$ is
expectation with respect to $\mathbb{P}_w$.

Assume that $\sigma$ is a random  up-right path chosen (according to some
distribution) from all paths inside a given subset
$\Sigma_{M,N}\subseteq\square_{M,N}$. A simple example is when
$\Sigma_{M,N}=\square_{M,N}$ and $\sigma$ is chosen
uniformly from all $\binom{M+N-2}{M-1}$ possible paths. In any case, $\sigma$ is a
random path defined on some
$(\Omega_\sigma,\mathcal{F}_\sigma,\mathbb{P}_\sigma)$, and expectation with
respect to $\mathbb{P}_{\sigma}$ will be denoted $\mathbb{E}_{\sigma}$.

We thus think of the energy of a path $\langle Y^{\sigma},w \rangle$ in
\eqref{weight_of_path} as a random variable defined on
$(\Omega_w\times\Omega_\sigma,\mathcal{F}_w\times\mathcal{F}_\sigma,\mathbb{P}_w\times\mathbb{P}_\sigma)$
which depends on both the randomness in the environment and in the path 
(the path is independent from the environment). The
goal of this paper is to show that for certain natural distributions of
$\sigma$ and $\mathbb{P}_w$-almost every environment, the (quenched) random variable
$\langle Y^{\sigma},w \rangle$ depending on the random path $\sigma$ is
asymptotically Gaussian as $M,N\to\infty$. A precise formulation
is given next.

\subsection{Quenched central limit theorems}
\label{sub:gaussian_concentration_intro}
Let the environment $w=\left\{ w_{ij} \right\}$ consist of independent random
variables satisfying \eqref{environment_moment_assumptions}.  Let $M=\lfloor
\xi N \rfloor $, where $\xi>0$ is fixed.

\begin{theorem}
	\label{thm:clt_all}
 If $\mathsf{p}>12$ and $\sigma$ is chosen uniformly from all up-right paths
 in the rectangle $\square_{M,N}$, then $\mathbb{P}_w$-almost surely,
	\begin{equation}
		\label{eq:clt_all}
		\frac{1}{\sqrt{M+N-1}}\sum_{(i,j)\in \sigma}w_{ij}
		\xrightarrow{\ \mathscr{D}\ }\mathscr{N}(0,1)
		,\qquad
		N\to\infty ,
	\end{equation}
where the convergence in distribution to the standard normal is with respect
to the marginal $\mathbb{P}_\sigma$. 
\end{theorem}

In other words, for large $N$ and almost every environment $w$, the empirical
distribution of the family of path energies will be approximately Gaussian.
The assumption $\mathsf{p}>12$ may seem unexpected. As will be seen in the
proofs, it arises from a combination of the moment assumptions needed for our
concentration inequality and our path counting estimates. 

\begin{corollary}
	\label{thm:clt_conditioned}
  The statement of \Cref{thm:clt_all} remains valid (still with
  $\mathsf{p}>12$) when $\sigma$ is chosen uniformly from all up-right paths
  passing through each of the points $(\lfloor \xi_i  N\rfloor , \lfloor
  \zeta_i N \rfloor )$, $i=1,\ldots, \ell$ (for finite~$\ell$), where
  $0<\xi_1<\ldots<\xi_\ell<\xi $ and $0<\zeta_1<\ldots<\zeta_\ell <1$ are
  fixed.  See \Cref{fig:clt_regimes},~(a).
\end{corollary}

\Cref{thm:clt_all} and \Cref{thm:clt_conditioned}
are proven in \Cref{sub:clt_all}.

To illustrate that our approach can yield similar results for other path
families, as long as suitable path counting arguments are available, we will
also prove the analogous result for the family of paths avoiding a hole of
fixed proportion in the center of $\square_{M,N}$. For simplicity in
\Cref{thm:clt_hole} below we assume  that $M=N$ (though a suitably modified
statement can be established for $M=\lfloor \xi N \rfloor $ as well). Let
$B=\lfloor \beta N \rfloor $, where $\beta\in(0,1)$ is fixed, and define the
subset $\Sigma_{N,N}=\left\{ (i,j)\in\square_{N,N}\colon \max\left(
|i-N/2|,|j-N/2| \right)\ge B/2 \right\}$; see \Cref{fig:clt_regimes},~(b).

\begin{theorem}
	\label{thm:clt_hole}
  If $\mathsf{p}>12$, and $\sigma$ be chosen uniformly from the set of up-right
	paths that remain in $\Sigma_{N,N}$, then $\mathbb{P}_w$-almost surely,
	\begin{equation}
		\label{eq:clt_hole}
		\frac{1}{\sqrt{2N-1}}\sum_{(i,j)\in \sigma}w_{ij}
		\xrightarrow{\ \mathscr{D}\ }\mathscr{N}(0,1)
		,\qquad
		N\to\infty.
	\end{equation}
\end{theorem}

This theorem is proven in \Cref{sub:clt_hole}.

\begin{remark}
	\label{rmk:nonzero_mean}
	In \eqref{environment_moment_assumptions}
	we assumed that the environment random variables $w_{ij}$ have mean $0$ and variance $1$.
	By shifting and scaling 
	the energies $\langle Y^{\sigma},w \rangle$
	one readily sees that
	all our results formulated above can be extended to 
	independent
	environments with arbitrary constant mean and nonzero variance.
\end{remark}

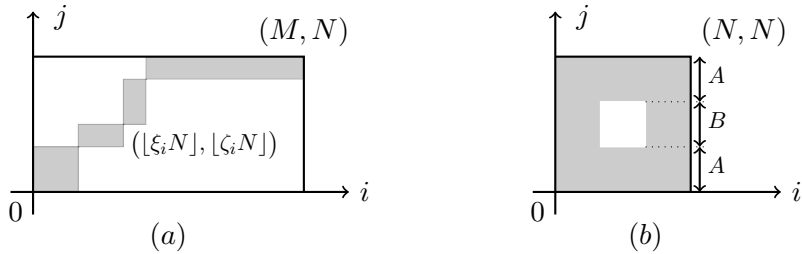
\begin{figure}[htpb]
	\centering
	\begin{tabular}{ccc}
		\begin{tikzpicture}[scale=.6]
			\draw[->, thick] (-.5,0)--++(7.5,0) node [right]
			{$i$};
			\draw[->, thick] (0,-.5)--++(0,4.5) node [right,
				xshift=4, yshift=-2] {$j$};
			\draw[thick] (6,0)--++(0,3)--++(-6,0);
			\draw[fill] (6,3) circle (0pt) node[above] {$(M,N)$};
			\draw[fill] (0,0) circle (0pt) node[below left] {$0$};
			\draw[fill] (1,1) circle (0pt);
			\draw[fill] (2,1.5) circle (0pt);
			\draw[fill] (2.5,2.5) circle (0pt);
			\draw[fill, opacity=.2]
			(0,0)--++(1,0)--++(0,1)--++(-1,0)--cycle;
			\draw[fill, opacity=.2]
			(1,1)--(2,1)--(2,1.5)--(1,1.5)--cycle;
			\draw[fill, opacity=.2]
			(2,1.5)--(2.5,1.5)--(2.5,2.5)--(2,2.5)--cycle;
			\draw[fill, opacity=.2]
			(2.5,2.5)--(2.5,3)--(6,3)--(6,2.5)--cycle;
			\node at (3,-1) {$(a)$};
			\node at (3.8,1.05) {\scalebox{.8}{$\bigl(\lfloor
						\xi_i N
						\rfloor ,\lfloor \zeta_i N
						\rfloor \bigr)$}};
		\end{tikzpicture}
		 & \hspace{40pt}
		\begin{tikzpicture}[scale=.6]
			\draw [fill,opacity=.2]
			(0,0)--++(3,0)--++(0,3)--++(-3,0)--cycle;
			\draw [fill,white]
			(1,1)--++(1,0)--++(0,1)--++(-1,0)--cycle;
			\draw[->, thick] (-.5,0)--++(5.5,0) node [right]
			{$i$};
			\draw[->, thick] (0,-.5)--++(0,4.5) node [right,
				xshift=4, yshift=-2] {$j$};
			\draw[thick] (3,0)--++(0,3)--++(-3,0);
			\draw[fill] (3,3) circle (0pt) node[above right]
				{$(N,N)$};
			\draw[fill] (0,0) circle (0pt) node[below left] {$0$};
			\draw[<->, thick] (3.2,0)--++(0,1) node[right,
				yshift=-7pt, xshift=-1pt] {\scalebox{.8}{$A$}};
			\draw[<->, thick] (3.2,1)--++(0,1) node[right,
				yshift=-7pt, xshift=-1pt] {\scalebox{.8}{$B$}};
			\draw[<->, thick] (3.2,2)--++(0,1) node[right,
				yshift=-7pt, xshift=-1pt] {\scalebox{.8}{$A$}};
			\draw[dotted] (2,1)--++(1.4,0);
			\draw[dotted] (2,2)--++(1.4,0);
			\node at (2,-1) {$(b)$};
		\end{tikzpicture}
	\end{tabular}
	\caption{Subsets of the rectangle from which the up-right path
		$\sigma$ is chosen uniformly in (a) \Cref{thm:clt_conditioned}, (b)
		\Cref{thm:clt_hole}.}
	\label{fig:clt_regimes}
\end{figure}

\subsection{Relation to other models}

We now briefly compare our setting with other models based on directed
up-right paths in a random environment. Consider first the directed polymer
models introduced in statistical physics in \cite{HuseHenleyPolymers1985}; see
also e.g.\ \cite{ImbrieSpencer1987}, \cite{Seppalainen2012}. In the lattice
setting, the directed polymer partition function is defined as (we continue to
assume that $M=N$)
\begin{equation*}
	Z_{N}(w) = \sum_{\sigma}\exp\Bigl\{ \beta\sum_{(i,j)\in \sigma}w_{ij}
	\Bigr\} = \sum_{\sigma}\exp\bigl\{\beta\,\langle Y^{\sigma},w \rangle\bigr\},
\end{equation*}
where the outer sum is taken over all up-right paths $\sigma$ inside
$\square_{N,N}$, and $\beta>0$ is the inverse temperature. The polymer weight
of a path $\sigma$ is defined as
\begin{equation*}
	Q_N(\sigma;w)= \frac1{Z_{N}(w)} \exp\bigl\{\beta\,\langle
	Y^{\sigma},w \rangle\bigr\}.
\end{equation*}
The study of the asymptotic behavior of $Z_N$ and $Q_N$ as $N\to\infty$ has
received a lot of attention in the past 30 years. Of particular interest are
the asymptotic fluctuations of the free energy $\log Z_N(w)$. These
fluctuations are expected to grow as $N^{1/3}$ under mild assumptions.
However, this scaling behavior is currently known only for a number of
integrable cases (that is for special choices of the distribution of $w_{ij}$
leading to exact formulae for the Laplace transform of $Z_N$);
see \cite{OConnellOrtmann2014}, \cite{BorodinCorwin2011Macdonald},
\cite{BorodinCorwinFerrari2012}. In integrable cases, the fluctuations
themselves are governed by one of the Tracy-Widom distributions
\cite{tracy_widom1994level_airy}, which is characteristic for the
Kardar-Parisi-Zhang universality \cite{CorwinKPZ}. Study of the asymptotic
fluctuations of $\log Z_N(w)$ when the integrability is not known presents a
major open problem in the field.

\medskip

Passing to the zero temperature limit $\beta\to\infty$ turns the free energy
$\log Z_N(w)$ into the last passage percolation time:
\begin{equation}
	\label{eq:LPP_time}
	G_N(w)=\max_{\sigma}\,\langle Y^{\sigma},w \rangle,
\end{equation}
where the maximum is taken over all up-right paths inside $\square_{N,N}$.
Assume that the environment variables $w_{ij}$ are nonnegative. This does not
significantly restrict generality since if the distribution of $w_{ij}$ is
bounded from below, one can achieve nonnegativity by adding a fixed constant
to all the $w_{ij}$. We can then interpret the nonnegative $w_{ij}$ as
random waiting times in the corner growth model so that \eqref{eq:LPP_time}
becomes the time at which the growing interface covers $(N,N)$. For further
details on corner growth we refer to
\cite{johansson2000shape}, 
\cite{seppalainen2009lecture}, 
\cite{CorwinKPZ},
\cite{baik2016-book},
\cite{seppalainen2017variational}.

Asymptotic fluctuations of $G_N(w)$ in integrable cases (when the $w_{ij}$'s
have exponential or geometric distribution) have been shown to converge on
scale $N^{1/3}$ to the Tracy-Widom distribution \cite{johansson2000shape}.
Again, the problem of asymptotic fluctuations in the corner growth model with
other distributions of $w_{ij}$ (for which exact formulae are not known to
exist) is open.

\medskip

Our results (in particular, \eqref{eq:clt_all}) mean that that the path
energies $\langle Y^\sigma,w\rangle$ asymptotically behave as
$(2N-1)\mathbb{E}w_{ij}+\zeta\sqrt {2N-1}$ (with $\zeta$
Gaussian). Let us compare this with the order of $G_N(w)$ known exactly for
special distributions of $w_{ij}$ from, e.g., \cite{johansson2000shape}. (A
similar comparison may be performed in the polymer case, but we omit it.) When
$w_{ij}$ are geometric, $\mathbb{P}(w_{ij}=k)=(1-q)q^k$, $k\in
\mathbb{Z}_{\ge0}$, we have (ignoring fluctuations) $\mathbb{E}G_N(w)\sim
N\frac{2\sqrt q}{1-\sqrt q}$, whereas typical values of  $\langle
Y^\sigma,w\rangle$ are of order $(2N-1)\mathbb{E}w_{ij}\sim N\frac{2q}{1-q}$,
which is smaller (however, the difference goes to zero as $q\searrow 0$).
Similarly, for $w_{ij}$ exponential with mean $1$, the last passage time
behaves as $\mathbb{E}G_N(w)\sim 4N$, and typical values of  $\langle
Y^\sigma,w\rangle$ are of order $(2N-1)\mathbb{E}w_{ij}\sim 2N$.

Therefore, while our results indicate that the asymptotic quenched behavior of
the typical values of the path energies $\langle Y^\sigma,w\rangle$ is
universal (i.e., does not depend on the distribution of the $w_{ij}$ under
mild assumptions), this conclusion does not extend to extreme values of
$\langle Y^\sigma,w\rangle$ responsible for the asymptotics of the last
passage time $G_N(w)$.

In \Cref{sec:concentration} we employ Talagrand's concentration inequality to
establish a quenched central limit theorem (\Cref{thm:general_QCLT}) modulo an
estimate on the distribution of the up-right path~$\sigma$. In
\Cref{sec:path_counting} we obtain the needed combinatorial estimates for
natural ensembles of up-right paths described in
\Cref{sub:gaussian_concentration_intro} above, and complete the proofs of
\Cref{thm:clt_all}, \Cref{thm:clt_hole}, and \Cref{thm:clt_conditioned}.

\medskip\noindent {\bf Acknowledgements.} 
We appreciate helpful discussions with 
Alexander Drewitz, 
Elizaveta Rebrova,
Timo
Sepp\"a\-l\"ainen, and Roman Vershynin. 
Part of this work was completed when MM and LP 
participated at a workshop 
``Analytic Tools in Probability and Applications'' 
at IMA in 2015, and we are grateful to the Institute and the workshop's 
organizers for hospitality and support.
MM is partially supported by
grant \#315593 from the Simons Foundation. LP is partially supported
by the NSF grant DMS-1664617.

\section{Gaussian concentration and quenched Central Limit Theorems}
\label{sec:concentration}

In this section we focus on general concentration estimates, and establish our
quenched CLT modulo combinatorial estimates which are postponed until
\Cref{sec:path_counting}.

\subsection{General concentration lemma}

Let us work in a more general setting. Suppose $\Sigma$ is a set with $n$
elements equipped with independent weights $\left\{ w_a\mid a\in \Sigma
	\right\}$ satisfying conditions \eqref{environment_moment_assumptions} for
  $\mathsf{p}\ge 3$. For a fixed $R>0$ we define the truncations
\begin{equation}
	\label{truncations_definition}
	w_a^{(R)}:=w_a\mathbf{1}_{|w_a|\le R}.
\end{equation}
Let $\sigma$ be a random subset of $\Sigma$ having almost surely $m$ elements.
As before, $\mathbb{P}_w$ and $\mathbb{P}_\sigma$ stand for the marginal
probability measures corresponding to $\left\{ w_a \right\}$ and $\sigma$,
respectively, and similarly for expectations. Let
$Y^{\sigma}_a:=\mathbf{1}_{a\in\sigma}$, and let $Y^\sigma$ denote the
corresponding random vector in $\mathbb{R}^n$.  Define
\begin{equation}
	\label{L_constant_abstract}
	L := \bigl\| \mathbb{E}_\sigma Y^\sigma \bigr\|_2 =
	\bigg(\sum_{a\in\Sigma}\mathbb{P}_\sigma\left( a\in\sigma
	\right)^2\biggr)^{\frac{1}{2}}.
\end{equation}

Next, let $\gamma\sim \mathscr{N}(0,1)$ denote the standard Gaussian measure
on $\mathbb{R}$, and let $\mu_{w}$ and $\mu_{w}^{(R)}$ be the quenched
distributions (conditioned on $w$) of the random variables
\begin{equation}
	\label{mu_definitions}
	\frac{\langle Y^\sigma,w \rangle }{\sqrt m}=\frac{1}{\sqrt
		m}\sum_{a\in\sigma}w_a,
	\qquad
	\frac{\langle Y^\sigma,w^{(R)} \rangle }{\sqrt m}=\frac{1}{\sqrt
		m}\sum_{a\in\sigma}w_a^{(R)},
\end{equation}
respectively. In particular, for a test function $f\colon
	\mathbb{R}\to\mathbb{R}$ we have
\begin{equation*}
	\int f\,d\mu_{w}
	=
	\mathbb{E}_{\sigma}f\left( \frac{\langle Y^\sigma,w \rangle }{\sqrt m}
	\right)
	=
	\mathbb{E}_{\sigma}f\left( \frac{1}{\sqrt m}\sum_{a\in\sigma}w_a
	\right)
	.
\end{equation*}

\begin{lemma}
	\label{lemma:general_concentration_lemma}
	Under the above assumptions, there exist absolute constants $C,c>0$ (not
  depending on parameters of the model) and a constant $\kappa>0$ depending only on $\mathsf{K}$ and
	$\mathsf{p}$ in \eqref{environment_moment_assumptions} such that for any
  $s,t,R>0$ and any convex	1-Lipschitz\footnote{Recall that a function $f\colon \mathbb{R}\to\mathbb{R}$
		is called 1-Lipschitz if $|f(x)-f(y)|\le|x-y|$ for all $x,y\in\mathbb{R}$.}
	function $f\colon \mathbb{R}\to \mathbb{R}$ we
	have
	\begin{equation}
		\label{general_concentration_lemma}
		\mathbb{P}_{w}
		\left[
			\left|
			\int f\,d\mu_{w}
			-
			\int f\,d\gamma
			\right|
			\ge
			\frac{\kappa}{\sqrt m}+
			\sqrt{\frac{\mathsf{K}nL^2}{mR^{\mathsf{p}-2}}}+s+t
			\right]
		\le
		\frac{\mathsf{K}nL^2}{mR^{\mathsf{p}-2}s^2}+C\exp\left[
			-c\,\frac{mt^2}{L^2R^2} \right].
	\end{equation}
	
\end{lemma}
\begin{proof}
	We have
	\begin{align}
		\label{general_concentration_lemma_proof1}
		\left|\int f\,d\mu_{w}-\int f\,d\gamma\right| & \le
		\left|\int f\,d\mu_{w}-\int f\,
		d\mu_{w}^{(R)}\right|
		\\&\hspace{10pt}+
		\label{general_concentration_lemma_proof2}
		\left|\int f\, d\mu_{w}^{(R)}-\mathbb{E}_w\int f\,
		d\mu_{w}^{(R)}\right|
		\\&\hspace{10pt}+
		\label{general_concentration_lemma_proof3}
		\left|\mathbb{E}_w\int f\, d\mu_{w}^{(R)}-\mathbb{E}_w\int f\,
		d\mu_{w}\right|
		\\&\hspace{10pt}+
		\label{general_concentration_lemma_proof4}
		\left|\mathbb{E}_w\int f\, d\mu_{w}-\int f\,d\gamma\right|.
	\end{align}
	The terms \eqref{general_concentration_lemma_proof1} and
	\eqref{general_concentration_lemma_proof3} on the right will
	be estimated by elementary methods, \eqref{general_concentration_lemma_proof2}
	via Talagrand's concentration inequality for independent bounded random
	variables, and \eqref{general_concentration_lemma_proof4} via an appropriate
	version of the central limit theorem. Note that the terms
	\eqref{general_concentration_lemma_proof3} and
	\eqref{general_concentration_lemma_proof4} are deterministic.

	We start with the last term
	\eqref{general_concentration_lemma_proof4}. 
	As a consequence of results of Esseen \cite{Esseen1958}, 
	we have a bound on Wasserstein-1
	distance that is independent of the path $\sigma$; 
	see for example \cite[Proposition 2.2]{goldstein2007l1Bounds}
	for a statement of this bound that depends only on finite third
	moments (as we assume here). 
	Note also that the $L_1$-distance of distribution
	functions as used in \cite{goldstein2007l1Bounds} is equivalent to 
	the Wassertein-1 distance in our setting since we use 
	1-Lipschitz test functions.
	For all $\sigma$ (recall that $\sigma$ has $m$ elements a.s.) we have
	\begin{equation}
		\label{general_concentration_lemma_proof_next1}
		\left|
		\mathbb{E}_w f\left( \frac{1}{\sqrt m}\sum_{a\in \sigma}w_a
		\right)
		-
		\int f\, d\gamma
		\right|\le \frac{\kappa}{\sqrt m},
	\end{equation}
	where $\kappa>0$ depends only on the third absolute moment of $w_a$
	and can therefore be bounded in terms of the constants $\mathsf{p}$ and
	$\mathsf{K}$ from \eqref{environment_moment_assumptions}. Since
	\eqref{general_concentration_lemma_proof_next1} holds for all $\sigma$, an
  application of Fubini's theorem implies that
	\begin{equation}
		\label{general_concentration_lemma_proof_next2}
		\left|
		\mathbb{E}_w\int f\,d\mu_{w}
		-
		\int f\,d\gamma
		\right|
		=
		\left|
		\mathbb{E}_w\mathbb{E}_\sigma f
		\left( \frac{1}{\sqrt m}\sum_{a\in\sigma}w_a \right)
		-
		\int f\,d\gamma
		\right|
		\le
		\frac{\kappa}{\sqrt m}.
	\end{equation}
	(In particular, this implies that the expectation $\mathbb{E}_w\int f \, d\mu_w$ 
	in \eqref{general_concentration_lemma_proof4}
	is finite.)

	Let us turn to \eqref{general_concentration_lemma_proof2}. Consider
	the function $F\colon \mathbb{R}^{\Sigma}\to \mathbb{R}$ defined by
	\begin{equation*}
		F(w):=\int f\,d\mu_w=\mathbb{E}_\sigma f\left( \frac{\langle
			Y^\sigma,w \rangle }{\sqrt m} \right).
	\end{equation*}
	Since $f$ is convex, $F$ is an average of convex functions on
	$\mathbb{R}^{\Sigma}$, and so is also convex. Let us now estimate the
	Lipschitz constant of $F$. We have for $w,w'\in \mathbb{R}^{\Sigma}$:
	\begin{align}
		\nonumber
		|F(w)-F(w')|
		 & \le
		\mathbb{E}_\sigma
		\left|
		f\left( \frac{\langle Y^\sigma,w \rangle }{\sqrt m}\right)
		-
		f\left( \frac{\langle Y^\sigma,w' \rangle }{\sqrt m} \right)
		\right|
		\\&\le
		\nonumber
		\frac{1}{\sqrt m}
		\mathbb{E}_\sigma
		\left|\langle Y^\sigma,w \rangle -\langle Y^\sigma,w' \rangle
		\right|
		\\&\le
		\nonumber
		\frac{1}{\sqrt m}
		\mathbb{E}_\sigma
		\left\langle Y^\sigma,|w-w'|\right\rangle
		\\&=
		\nonumber
		\frac{1}{\sqrt m}
		\left\langle \mathbb{E}_\sigma Y^\sigma,|w-w'|\right\rangle
		\\&\le \frac{1}{\sqrt m}L \|w-w'\|_{2}.
		\label{general_concentration_lemma_proof_Lip_estimate}
	\end{align}
	Here the third line follows because $Y^\sigma$ has nonnegative
	components, $|w-w'|$ denotes componentwise absolute value, and the last line
	follows from the Cauchy--Schwarz inequality. Thus, we see that $F$ has
	Lipschitz constant at most $L/\sqrt m$.

	Using this and applying Talagrand's inequality for bounded independent
	random variables \cite{talagrand1995concentration}\footnote{See also \cite[Corollary 1.3]{ledoux1997talagrand} 
	and references and discussion therein.}
	to $F(w^{(R)})=\int
		f\,d\mu_w^{(R)}$ we obtain for any $t>0$ and for some absolute constants
    $C,c>0$,
	\begin{equation}
		\label{general_concentration_lemma_proof_next3}
		\mathbb{P}_w\left[\,
			\left|
			\int f\,d\mu_w^{(R)}-\mathbb{E}_w\int f\,d\mu_w^{(R)}
			\right|\ge t
			\right]\le
		C\exp\left[ -c\frac{mt^2}{L^2R^2} \right].
	\end{equation}

	It remains to consider terms
	\eqref{general_concentration_lemma_proof1} and
	\eqref{general_concentration_lemma_proof3}. From the Lipschitz estimate
	\eqref{general_concentration_lemma_proof_Lip_estimate} we have
	\begin{equation*}
		\left|
		\int f\,d\mu_w-\int f\,d\mu_w^{(R)}
		\right|
		=
		\bigl|
		F(w)-F(w^{(R)}) \bigr| \le \frac{L}{\sqrt m} \bigl\| w-w^{(R)}
		\bigr\|_2 = \frac{L}{\sqrt m} \left(
		\sum_{a\in\Sigma}w_a^2\mathbf{1}_{|w_a|>R} \right)^{\frac{1}{2}}.
	\end{equation*}
	Utilizing H\"older and Chebyshev inequalities, we can write
	\begin{equation*}
		\mathbb{E}_w \bigl(w_a^2\mathbf{1}_{|w_a|>R}\bigr)
		\le
		\bigl(\mathbb{E}_w
		|w_a|^{\mathsf{p}}\bigr)^{\frac{2}{\mathsf{p}}}
		\left(
		\mathbb{P}_{w}\bigl[ |w_a|>R \bigr]
		\right)
		^{1-\frac{2}{\mathsf{p}}}
		\le
		\frac{\mathbb{E}_w|w_a|^{\mathsf{p}}}
		{R^{\mathsf{p}(1-2/\mathsf{p})}}\le
		\frac{\mathsf{K}}{R^{\mathsf{p}-2}},
	\end{equation*}
	and so for any $u>0$ we have
	\begin{equation*}
		\mathbb{P}_w
		\bigg[
			\sum_{a\in\Sigma}w_a^2\mathbf{1}_{|w_a|>R}\ge u
			\biggr]
		\le
		\frac{n\mathsf{K}}{uR^{\mathsf{p}-2}}.
	\end{equation*}
	Therefore, we have
	\begin{equation}
		\label{general_concentration_lemma_proof_next4}
		\left|
		\mathbb{E}_w\int f\,d\mu_w-\mathbb{E}_w\int f\,d\mu_w^{(R)}
		\right|
		\le
		L \sqrt{\frac{\mathsf{K}n}{mR^{\mathsf{p}-2}}},
	\end{equation}
	and
	\begin{equation}
		\label{general_concentration_lemma_proof_next5}
		\mathbb{P}_w
		\left[ \,
			\left|
			\int f\,d\mu_w-\int f\,d\mu_w^{(R)}
			\right|\ge s
			\right]\le
		\frac{n\mathsf{K}L^2}{mR^{\mathsf{p}-2}s^2}.
	\end{equation}

	The lemma now follows by combining
	\eqref{general_concentration_lemma_proof_next2},
	\eqref{general_concentration_lemma_proof_next3},
	\eqref{general_concentration_lemma_proof_next4}, and
	\eqref{general_concentration_lemma_proof_next5}.
\end{proof}

\subsection{Quenched central limit theorem}

Our aim now is to apply the general \Cref{lemma:general_concentration_lemma}
to obtain the quenched central limit theorem in the corner growth setting
described in \Cref{sec:intro}. Recall that we choose $N$ (the vertical
dimension of the rectangle in \Cref{fig:coordinates_paths}) to be the main
parameter going to infinity, and we let $M=\lfloor \xi N \rfloor$ for some
$\xi>0$. The parameters in \Cref{lemma:general_concentration_lemma} are
instantiated as follows:
\begin{equation}
	\label{parameters_specialized}
	n(N)=|\Sigma_{M,N}|\le MN\sim \xi N^2, \qquad m(N)=M+N-1\sim (\xi+1)N,
	\qquad L=L(N).
\end{equation}
Here we assume that for each $N$, $\Sigma_{M,N}\subseteq\square_{M,N}$ is a
given subset and that $\sigma$ is chosen according to some distribution such
that $\sigma\in\Sigma_{M,N}$ almost surely (it replaces the set $\Sigma$ in
\Cref{lemma:general_concentration_lemma}). We further assume that as
$N\to\infty$,
\begin{equation}
	\label{general_QCLT_power_assumptions}
	\begin{split}
		n(N)&=|\Sigma_{M,N}|=O(N^{\eta}), \\ L(N)&=
		\bigg(\sum_{a\in\Sigma_{M,N}}\mathbb{P}_\sigma\left( a\in\sigma
		\right)^2\biggr)^{\frac{1}{2}} =O(N^{\lambda}),
	\end{split}
\end{equation}
for some $0<\eta\le 2$ and $0<\lambda\le \eta/2$.\footnote{The fact that we must
  have $\lambda\le \eta/2$ follows by taking the trivial estimate
	$\mathbb{P}_\sigma(a\in \sigma)\le 1$ for all $a$.}
  For the specific subsets $\Sigma_{M,N}$ considered in this paper, the parameter $L(N)$ will be
estimated separately in \Cref{sec:path_counting} below. Note that the
constants $C,c,\mathsf{K}, \mathsf{p}, \kappa$ in
\Cref{lemma:general_concentration_lemma} are independent of $N$.

\begin{theorem}
	\label{thm:general_QCLT}
	Under the above assumptions, if $\lambda<\frac{1}{2}$ and
	$\mathsf{p}>\frac{6}{1-2\lambda}$ then $\mathbb{P}_w$-almost surely,
	\begin{equation*}
		\frac{1}{\sqrt{m(N)}}\sum_{a\in \sigma}w_a\xrightarrow{\
			\mathscr{D}\ }\mathscr{N}(0,1),
		\qquad N\to\infty,
	\end{equation*}
	where the convergence in distribution to the standard normal is with respect to the
  marginal $\mathbb{P}_\sigma$.
\end{theorem}
\begin{proof}
	To get the desired $\mathbb{P}_w$-almost sure convergence in
	distribution, we will choose the parameters $R,s,t$ in
	\Cref{lemma:general_concentration_lemma} depending on $N$ and apply the
	Borel-Cantelli lemma.\footnote{Convex 1-Lipschitz test functions are 
	enough to conclude convergence in distribution. 
	First, we have tightness since the first moments converge. 
	By the Weierstrass theorem, 
	compactly supported test functions can be approximated
	by polynomials (on a compact set), 
	and polynomials are linear combinations of 
	convex 1-Lipschitz functions. 
	See also, e.g.,
	\cite{guionnet2000concentration}, \cite{meckes2009some}
	for slightly different approaches.} 
	That is, from the left side of
	\eqref{general_concentration_lemma} we see that we must have
	\begin{equation*}
		\lim_{N\to\infty}
		\left[
			\frac{\kappa}{\sqrt{\xi+1}}\frac{1}{\sqrt N}+
			\sqrt{\frac{\mathsf{K}}{(\xi+1)}}
			\frac{\sqrt{n(N)}L(N)}{\sqrt N
				R(N)^{\frac{\mathsf{p}-2}{2}}}
			+s(N)+t(N)
			\right]=0,
	\end{equation*}
	which is equivalent to
	\begin{equation}
		\label{general_QCLT_proof1}
		\lim_{N\to\infty}
		\frac{n(N)L(N)^2}{N R(N)^{\mathsf{p}-2}}
		=0,\qquad
		\lim_{N\to\infty}s(N)=
		\lim_{N\to\infty}t(N)=0.
	\end{equation}
	Moreover, to use Borel-Cantelli the right side of \eqref{general_concentration_lemma}
	must be summable, which is equivalent to
	\begin{equation}
		\label{general_QCLT_proof2}
		\sum_{N= N_0}^{\infty}
		\frac{n(N)L(N)^2}{NR(N)^{\mathsf{p}-2}s(N)^2}<\infty,\qquad
		\sum_{N= N_0}^{\infty}
		\exp\left[ -c(\xi+1)\,\frac{Nt(N)^2}{L(N)^2R(N)^2} \right]<\infty,
	\end{equation}
	for some absolute constant $N_0$.

	Let $R(N)\sim N^{\rho}$ for some $\rho>0$. Then for
	\eqref{general_QCLT_proof1} and \eqref{general_QCLT_proof2} to hold under our
	assumption \eqref{general_QCLT_power_assumptions} it is
	necessary and sufficient that
	\begin{equation}
		\label{general_QCLT_proof_counting_powers}
		\eta+2\lambda-1-\rho(\mathsf{p}-2)<-1,\qquad
		1-2\lambda-2\rho>0,
	\end{equation}
	and that $s(N)$ and $t(N)$ tend to zero sufficiently slowly as negative
	powers of $N$. Setting $\rho=\frac{1}{2}-\lambda-\varepsilon$ for small enough
	$\varepsilon>0$, one can check that condition
	\eqref{general_QCLT_proof_counting_powers} (together with our assumptions $\eta\le
	2$ and $\lambda\le \eta/2$ coming from \eqref{L_constant_abstract} and
	\eqref{parameters_specialized}) is equivalent to
	\begin{equation*}
		0<\lambda<\frac{1}{2},\qquad
		\mathsf{p}>\frac{2+2\eta}{1-2\lambda}.
	\end{equation*}
	The latter inequality holds if $\mathsf{p}>\frac{6}{1-2\lambda}$,
	which completes the proof.
\end{proof}

\section{Path counting}
\label{sec:path_counting}

Our goal in this section is to obtain estimates of $L(N)$ of the form
\eqref{general_QCLT_power_assumptions} with $\lambda<\frac{1}{2}$ for the
concrete families of up-right paths described in
\Cref{sub:gaussian_concentration_intro}. These estimates lead to quenched
central limit theorems.
Similar estimates have appeared in the 
context of random polymers before,
for example, see
\cite[Appendix A]{AlbertsKhaninQuastel2012}.

\subsection{All possible paths -- proofs of \Cref{thm:clt_all} and
	\Cref{thm:clt_conditioned}}
\label{sub:clt_all}

We start with the case when $\sigma$ is chosen uniformly at random from the
set of all possible up-right paths in the rectangle $\square_{M,N}$, so
$\Sigma_{M,N}=\square_{M,N}$. Let us slice the rectangle as follows:
\begin{equation}
	\square_{M,N}
	=
	\bigsqcup_{k=2}^{M+N}\square_{M,N}^{(k)},
	\qquad
	\square_{M,N}^{(k)}
	:=
	\left\{ a=(i,j)\colon 1\le i\le M,\, 1\le j\le N,\, i+j=k \right\}.
	\label{rectangle_slicing}
\end{equation}
We can write
\begin{equation}
	\label{slicing_max_estimate}
	L(N)^2 = \sum_{k=2}^{M+N}\sum_{a\in\square_{M,N}^{(k)}}
	\mathbb{P}_\sigma\left(a\in\sigma \right)^2 \le \sum_{k=2}^{M+N}
	\max_{a\in\square_{M,N}^{(k)}}\mathbb{P}_\sigma\left( a\in\sigma \right) =
	\sum_{k=2}^{M+N} \mathscr{M}_k,
\end{equation}
where we have denoted
$\mathscr{M}_k:=\max_{a\in\square_{M,N}^{(k)}}\mathbb{P}_\sigma\left(
	a\in\sigma \right)$.
\begin{remark}
	\label{rmk:max_trivial_estimate}
	The estimate \eqref{slicing_max_estimate} holds for any distribution
	of the up-right path $\sigma$. Moreover, since the maximum probability over
	$\square_{M,N}^{(k)}$ can always be bounded by $1$, we have the trivial
	estimate $L(N)^2\le M+N-1$. In our regime ($M$ proportional to $N$) this
	estimate leads to $L(N)=O(N^{\frac{1}{2}})$ and so $\lambda=\frac{1}{2}$,
  which is not quite good enough for our purpose.
\end{remark}

We will show however that one can in fact take $\lambda=\frac{1}{4}$ using a
better estimate of $\mathscr{M}_k$. Recall that $M=\lfloor \xi N \rfloor $, $\xi>0$.

\begin{lemma}
	\label{lemma:all_paths}
	Let $\sigma$ be chosen uniformly from all possible paths in
	$\square_{M,N}$. There exists $\mathsf{C}>0$ such that for all $N$ large
	enough, $\sum_{k=2}^{M+N}\mathscr{M}_k\le \mathsf{C}\sqrt N$.
\end{lemma}
\begin{proof}
	Fix $k=2,\ldots,M+N $ and $a=(i,j)\in \square_{M,N}^{(k)}$, that is, $i+j=k$. Then
	\begin{equation}
		\label{hypergeometric_distribution}
		\mathbb{P}_\sigma\left( a\in \sigma \right)
		=
		\frac{\binom{i+j-2}{i-1}\binom{M+N-i-j}{M-i}}
		{\binom{M+N-2}{M-1}}
		=
		\frac{\binom{k-2}{i-1}\binom{M+N-k}{M-i}}
		{\binom{M+N-2}{M-1}}
		,
		\qquad
		i=1,\ldots,k-1 ,
	\end{equation}
	which is the hypergeometric distribution. 
	Indeed, the numerator counts pairs of paths from $(1,1)$ to $(i,j)$ and 
	from $(i,j)$ to $(M,N)$, and the denominator counts all possible paths.
	Let us denote the 
	probability \eqref{hypergeometric_distribution}
	by $p_i^{(k)}$. By looking at ratios $p_i^{(k)}/p_{i+1}^{(k)}$ and
	comparing this to $1$ one can readily see that the mode of the distribution
	\eqref{hypergeometric_distribution} (that is the $m\in\mathbb{R}$ such that
	$p_i^{(k)}$ achieves its maximum for an integer $i$ neighboring $m$) is
	\begin{equation*}
		m=\frac{(k-1) M}{M+N}.
	\end{equation*}
	Thus, plugging $i=m$ into \eqref{hypergeometric_distribution} leads to
	an upper bound, up to a constant,  on $\mathscr{M}_k$. We now consider three ranges
	of $k$. First, if $k\le \sqrt N$ or $k\ge M+N-\sqrt N$, then the number of
	summands in \eqref{slicing_max_estimate} is of order $\sqrt N$, and we
	estimate each of them by $1$. Next, let $k$ be from $\sqrt N$ to $\epsilon N$,
	where $\epsilon>0$ is fixed (in the corresponding interval close to $M+N$ the
	estimate will be similar). Then using Stirling's formula one can readily see
	that $p_m^{(k)}=O(1/\sqrt k)$. Summing $O(1/\sqrt k)$ over $k$ from $\sqrt N$
	to $\varepsilon N$ we get a term of order $\sqrt N$ as well. Finally, if $k$ is
	from $\epsilon N$ to $(1-\epsilon)(M+N)$, then Stirling's formula gives
	\begin{equation*}
		p_m^{(k)}=\frac{O(1)}{\sqrt
			N}\frac{1}{\sqrt{\frac{k}{N}(1+\xi-\frac{k}{N})}},
	\end{equation*}
	where $O(1)$ is independent of $k$. The sum of these expressions over
	our range of $k$ is $\sqrt N$ times a Riemann sum of a convergent integral,
	and thus also has order $\sqrt N$. This completes the proof.
\end{proof}
\begin{proof}[Proof of \Cref{thm:clt_all}]
	This theorem follows from \Cref{lemma:all_paths} and
	\Cref{thm:general_QCLT} with $\lambda=\frac{1}{4}$, which leads to the moment
	condition $\mathsf{p}>\frac{6}{1-2\lambda}=12$.
\end{proof}
\begin{proof}[Proof of \Cref{thm:clt_conditioned}]
	Let us now discuss the case when $\sigma$ is chosen uniformly from all
	up-right paths within $\square_{M,N}$ that pass through the points
	$(\lfloor \xi_i  N\rfloor , \lfloor \zeta_i N \rfloor )$, $i=1,\ldots, \ell$,
	where $0<\xi_1<\ldots<\xi_\ell<\xi $ and $0<\zeta_1<\ldots<\zeta_\ell <1$ are
	fixed. In this case the sum $\sum_{k=2}^{M+N}\mathscr{M}_k$ splits into
	$\ell+1$ sums. Each of these sums has $O(N)$ terms and similarly to
	\Cref{lemma:all_paths} one can show that each sum behaves as $O(\sqrt N)$.
	Thus, \Cref{thm:clt_conditioned} holds with the same moment condition
	$\mathsf{p}>12$ as in \Cref{thm:clt_all}.
\end{proof}
\begin{remark}
	The statement of \Cref{thm:clt_conditioned} continues to be valid if
	the number of points $\ell=\ell(N)$ through which the path $\sigma$ must pass
	goes to infinity, say, as $\ell(N)=O(N^{\alpha})$, $0<\alpha<1$. Indeed,
	assume in addition that the smaller rectangles as in
	\Cref{fig:clt_regimes},~(a) have asymptotically equivalent sides, and that the
	sides of all $O(N^\alpha)$ rectangles are also asymptotically equivalent. Then
	$\sum_{k=2}^{M+N}\mathscr{M}_k$ is bounded by $\mathsf{C}N^\alpha
		N^{\frac{1-\alpha}{2}}$, and so \Cref{thm:clt_conditioned} holds with
	$\lambda=\frac{1+\alpha}{4}$, which leads to the moment condition
	$\mathsf{p}>\frac{12}{1-\alpha}$.
\end{remark}

\subsection{Paths around a hole -- proof of \Cref{thm:clt_hole}}
\label{sub:clt_hole}

Define $\Sigma_{N,N}^{(k)}:=\Sigma_{N,N}\cap\square_{N,N}^{(k)}$, where
$\Sigma_{N,N}$ is the set of vertices with the hole removed, and
$\square_{N,N}^{(k)}$ is given by \eqref{rectangle_slicing}. Our goal is to
estimate
$\mathscr{M}_k=\max_{a\in\Sigma_{N,N}^{(k)}}\mathbb{P}_{\sigma}(a\in\sigma)$
to argue similarly to \Cref{sub:clt_all}.

\begin{lemma}
	\label{lemma:clt_hole}
	In the regime $B=\lfloor \beta N \rfloor $, $\beta\in(0,1)$, there
	exists $\mathsf{C}>0$ such that for all $N$ large enough we have
	$\sum_{k=2}^{2N}\mathscr{M}_k\le\mathsf{C}\sqrt N$.
\end{lemma}
\begin{proof}
	Let us denote $A:=(N-B)/2$, so $N=2A+B$, see
	\Cref{fig:clt_regimes},~(b). For simpler notation we will omit integer parts
	as this does not affect our up-to-constant estimates. In particular, we can
	and will assume that $A=\lfloor \frac{1-\beta}{2}N \rfloor $.

	By symmetry it suffices to assume that $j\ge i$ and $i+j\le N+1$; the
	estimates for other $(i,j)\in\Sigma_{N,N}$ would be the same. We have
	\begin{equation}
		\label{clt_hole_proof}
		\mathbb{P}_{\sigma}( (i,j)\in \sigma )
		=
		\frac{1}{Z_N}
		\sum_{y=1}^{A}
		\binom{k-2}{i-1}
		\binom{N+1-k}{y-i}\binom{N-1}{y-1}
		,\qquad (i,j)\in\Sigma_{N,N}^{(k)},
	\end{equation}
	where $Z_N=2\sum_{y=1}^{A}\binom{N-1}{y-1}^{2}$ is the number of
	up-right paths in $\Sigma_{N,N}$ (the factor $2$ comes from symmetry). Here
	$(y,N+1-y)$, $y=1,\ldots,A $, is the point where the up-right path intersects
	the line $i+j=N+1$. In \eqref{clt_hole_proof} we also used the convention that
	$\binom{N+1-k}{y-i}=0$ for $y<i$ since the first coordinate increases along up-right paths.

	Let us first maximize the quantity under the sum in
	\eqref{clt_hole_proof} in $i=1,\ldots,k-1 $ for fixed $y$. By considering the
	ratios of the terms with $i$ and $i+1$ we see that the mode in $i$ is at
	\begin{equation*}
		m(y)=\frac{y(k-1)}{N+1}.
	\end{equation*}
	Therefore, an up-to-constant upper bound for $\mathscr{M}_k$ following
	from \eqref{clt_hole_proof} is (for some $\mathsf{C}_1>0$)
	\begin{equation}
		\label{clt_hole_proof1}
		\mathscr{M}_k
		\le \frac{\mathsf{C}_1}{Z_N}
		\sum_{y=1}^{A}
		\binom{k-2}{m(y)-1}
		\binom{N+1-k}{y-m(y)}\binom{N-1}{y-1}.
	\end{equation}

	The sums over $y$ in both \eqref{clt_hole_proof1} and $Z_N$ are
	dominated by the behavior around $y=A$ because $\frac{A}{N}<\frac{1}{2}$.
	Indeed, this follows from standard large deviations type equivalences for the
	binomial coefficients:
	\begin{align*}
		\binom{N-1}{y-1}
		 & =
		O(N^{-\frac{1}{2}})\exp\left\{ -N
		\Bigl(\bigl(1-\frac{y}{N}\bigr) \log
		\bigl(1-\frac{y}{N} \bigr) +\frac{y}{N}  \log
		\bigl(\frac{y}{N} \bigr)\Bigr)+O(N^{-1})
		\right\}, \\
		\binom{k-2}{m(y)-1}
		 & =
		O(N^{-\frac{1}{2}})
		\exp
		\left\{
		-k
		\Bigl(
		\bigl(1-\frac{y}{N}\bigr) \log
		\bigl(1-\frac{y}{N} \bigr) +\frac{y}{N}  \log
		\bigl(\frac{y}{N} \bigr)
		\Bigr)+O(N^{-1})
		\right\}
		,         \\
		\binom{N+1-k}{y-m(y)}
		 & =
		O(N^{-\frac{1}{2}})
		\exp
		\left\{
		-(N-k)
		\Bigl(
		\bigl(1-\frac{y}{N}\bigr) \log
		\bigl(1-\frac{y}{N} \bigr) +\frac{y}{N}  \log
		\bigl(\frac{y}{N} \bigr)
		\Bigr)+O(N^{-1})
		\right\},
	\end{align*}
	and the fact that the function $\mathsf{y}\mapsto
		-(1-\mathsf{y})\log(1-\mathsf{y})-\mathsf{y}\log \mathsf{y}$ is positive and
	strictly increasing for $\mathsf{y}\in(0,1/2)$. In the above equivalences we
	assumed that $y/N$ is bounded away from $0$, and $k/N$ is bounded away from
	$0$ and $1$. The behavior at the tails can be estimated in a similar way;
  cf.\ the proof of \Cref{lemma:all_paths}.

	In the sums over $y=1,\ldots,A$, both in the numerator and the
	denominator in \eqref{clt_hole_proof1}, there are $O(\sqrt N)$ terms dominating
	the other terms. This implies that the right-hand side of
	\eqref{clt_hole_proof1} behaves as $C(k)N^{-\frac{1}{2}}$. It remains to see
	that the constant $C(k)$ coming from the numerator is summable over $k$ from
	$\epsilon N$ to $(1-\epsilon)N$. This constant can be computed using
	Stirling's approximation:
	\begin{equation*}
		C(k)= \frac{O(1)} {\sqrt{\frac{k}{N}(1-\frac{k}{N})}},
	\end{equation*}
	where $O(1)$ is independent of $k$ (but depends on $\beta$). The sum
	of these expressions over $k$ is equal to $N$ times the Riemann sum of a
	convergent integral. Therefore, the sum of the $\mathscr{M}_k$'s is bounded by
	$\sqrt N$, as desired.
\end{proof}

\Cref{thm:clt_hole} now follows from \Cref{lemma:clt_hole} and
\Cref{thm:general_QCLT}
with $\lambda=\frac{1}{4}$, so the moment condition is $\mathsf{p}>12$.

\bibliographystyle{amsalpha}
\bibliography{bib}

\end{document}